\newif\ifdraft
\newif\ifarxiv
\arxivtrue

\documentclass[reqno, 12pt]{article}
\usepackage[a4paper, margin=2cm]{geometry}

\ifdraft
\usepackage[inline]{showlabels}
\usepackage[fontsize=14pt]{scrextend}
\fi

\usepackage{amsmath, amsthm, thmtools, amssymb, mathtools}
\usepackage{enumerate}

\usepackage[dvipsnames]{xcolor}
\usepackage{hyperref}
\hypersetup{
	colorlinks = true,
	linkcolor = {BlueViolet},
	urlcolor={RoyalBlue},
	citecolor={BrickRed}
}
\usepackage{tikz-cd}
	
\makeatletter
\renewcommand{\itemautorefname}{\@gobble}
\makeatother

\usepackage{import}

\theoremstyle{plain}
\newtheorem{theorem}{Theorem}[section]
\newtheorem*{theorem*}{Theorem}

\newtheorem{prop}[theorem]{Proposition}
\newtheorem{corollary}[theorem]{Corollary}
\newtheorem{lemma}[theorem]{Lemma}

\theoremstyle{definition}
\newtheorem{defn}[theorem]{Definition}

\newtheorem*{conj*}{Conjecture}
\newtheorem{remark}[theorem]{Remark}

\title{Hausdorff Stability of the Cut Locus Under $C^2$-Perturbations of the Metric}

\author{
	Aritra Bhowmick\thanks{
		Department of Mathematics, Indian Institute of Science, Bengaluru, India.
		\href{avowmix@gmail.com}{\texttt{avowmix@gmail.com}}, 
		\href{aritrab@iisc.ac.in}{\texttt{aritrab@iisc.ac.in}}
	}
	\and
	Jin-ichi Itoh\thanks{
		School of Education, Sugiyama Jogakuen University, Nagoya, Japan.
		\texttt{\href{mailto:j-itoh@sugiyama-u.ac.jp}{j-itoh@sugiyama-u.ac.jp}}
	}
	\and
	Sachchidanand Prasad\thanks{
		School of Mathematics, Jilin University, China;
		Mathematisches Institut, G\"ottingen University, G\"ottingen, Germany
		\texttt{\href{mailto:sachchidanand.prasad1729@gmail.com}{sachchidanand.prasad1729@gmail.com}}
	}
}

\newcommand{\setsubjclass}[1]{\def\thesubjclass{#1}}
\newcommand{\setkeywords}[1]{\def\thekeywords{#1}}
\newcommand{\printclassification}{%
	\renewcommand{\thefootnote}{}%
	\footnotetext{\textbf{2020 Mathematics Subject Classification.} \thesubjclass.}%
	\footnotetext{\textbf{Keywords.} \thekeywords.}%
	\renewcommand{\thefootnote}{\arabic{footnote}}%
}

\setsubjclass{Primary: 53C22, 35D40; Secondary: 35F21, 49J52}
\setkeywords{cut locus, eikonal equation, viscosity solution, Hausdorff distance, injectivity radius}

\begin{document}
\allowdisplaybreaks
\setcounter{tocdepth}{3}
\frenchspacing 
\maketitle
\printclassification

\begin{abstract}
    In this article, we prove the stability with respect to the Hausdorff metric $d_H$ of the cut locus $\mathrm{Cut}(p, \mathfrak{g})$ of a point $p$ in a compact Riemannian manifold $(M, \mathfrak{g})$ under $C^2$ perturbation of the metric. Specifically, given a sequence of metrics $\mathfrak{g}_i$ on $M$, converging to $\mathfrak{g}$ in the $C^2$ topology, and a sequence of points $p_i$ in $M$, converging to $p$, we show that $\lim_i d_{H}\left( \mathrm{Cut}(p_i, \mathfrak{g}_i), \mathrm{Cut}(p, \mathfrak{g}) \right) = 0$. Along the way, we also prove the continuous dependence of the cut time map on the metric.
\end{abstract}

\section{Introduction}
The principal objects of study in Riemannian geometry are \emph{geodesics}, which are \emph{locally} distance minimizing curves. If $(M, \mathfrak{g})$ is a connected, complete Riemannian manifold, then between any two points, say, $p, q \in M$ there exists a geodesic which is \emph{globally} distance minimizing. We call such geodesics \emph{minimizers}. The cut locus $\mathrm{Cut}(p, \mathfrak{g})$ of $p$ consists of those points $q \in M$ beyond which a minimizer from $p$ to $q$ fails to be distance minimizing from $p$. Originally introduced by Poincar\`{e} \cite{poincareCutLocus}, the notion of cut locus has been extensively studied in the literature \cite{kobayashiCutLocus,buchnerAnalyticCutLocus,wolterCutLocus,sakaiBook}. A closely related set is $\mathrm{Sep}(p, \mathfrak{g})$, which consists of points $q$ admitting at least two distinct minimizers from $p$ to $q$. It is well known that $\mathrm{Cut}(p, \mathfrak{g})$ is the closure of $\mathrm{Sep}(p, \mathfrak{g})$, i.e., $\mathrm{Cut}(p, \mathfrak{g}) = \overline{\mathrm{Sep}(p, \mathfrak{g})}$, and consequently, $\mathrm{Cut}(p, \mathfrak{g})$ is a closed set of $M$. The concept of cut locus can be easily generalized to a closed submanifold (or even a closed set) of $M$, see \cite{basuPrasadCutLocus} for a survey of results.

The stability question of the cut locus has been studied in the literature as well. In \cite{buchnerCutLocusStability}, Buchner considered the notion of stability from the point of view of singularity theory \cite{guilleminGolubitsky}. Given $\mathcal{G}$ as the space of smooth Riemannian metrics on a compact manifold $M$, equipped with the Whitney $C^\infty$ topology, Buchner proved that $\mathrm{Cut}(p, \mathfrak{g})$ is stable for $\mathfrak{g}$ in an open dense subset $\mathcal{O} \subset \mathcal{G}$, provided $\dim M \le 6$. In \cite{stabilityMedialAxis}, the authors considered the stability of the set $\mathrm{Sep}^{\mathsf{in}}(\partial \Omega, \mathfrak{g}) \coloneqq \mathrm{Sep}(\partial \Omega, \mathfrak{g}) \cap \Omega$, called \emph{medial axis} therein, in the sense of the Hausdorff distance. They proved the Hausdorff stability of $\mathrm{Sep}^{\mathsf{in}}(\partial \Omega, \mathfrak{g})$ if the $C^2$-boundary $\partial \Omega$ is perturbed by $C^2$-small diffeomorphisms. In \cite{albanoCutLocusStability}, Albano also considered the Hausdorff stability of the cut locus. Given the collection $\mathcal{A}$ of bounded open domains $\Omega \subset \mathbb{R}^n$ with $C^2$-smooth boundary, and the collection $\mathcal{G}^\prime$ of $C^2$-smooth Riemannian metrics on $\mathbb{R}^n$, Albano considered the cut locus $\mathrm{Cut}^{\mathsf{in}}(\partial\Omega, \mathfrak{g}) \coloneqq \mathrm{Cut}(\partial \Omega, \mathfrak{g}) \cap \Omega$ of the boundary $\partial \Omega$ \emph{inside} $\Omega$, with respect to a metric $\mathfrak{g} \in \mathcal{G}^\prime$. Leveraging the properties of the viscosity solutions to the eikonal equations, it was proved that given a $C^2$-convergent sequence $(\Omega_i, \mathfrak{g}_i) \rightarrow (\Omega, \mathfrak{g})$, the cut locus $\mathrm{Cut}(\Omega_i, \mathfrak{g}_i)$ converges to $\mathrm{Cut}(\Omega, \mathfrak{g})$ in the Hausdorff metric; we refer to \cite{albanoCutLocusStability} for the notion of $C^2$-convergence in $\mathcal{A}$. In yet another direction, the topological stability of the cut locus is studied in \cite{cutLocusStabilityCLT} from the point of view of metric measure theory. 

In this article, we follow the approach of \cite{albanoCutLocusStability}, and extend it to compact Riemannian manifolds. In this context, we would like to mention that the viscosity solutions of the more general Hamilton-Jacobi equations were utilized in \cite{guijarroCutLocusCodim3} to get a characterization of the cut locus up to set of Hausdorff dimension $n-3$, and recently, similar techniques were also employed in \cite{cutLocusFrechetMean} where the authors used the notion of semiconcave functions to study the cut locus from a probabilistic point of view. \medskip

Let $M$ be a compact manifold, and $\mathcal{G}$ be the collection of smooth Riemannian metrics on $M$ with the Whitney $C^2$ topology. Then, for any $(p, \mathfrak{g}) \in M \times \mathcal{G}$, the cut locus $\mathrm{Cut}(p, \mathfrak{g})$, being a closed subset, is compact. In other words, we have a map $\mathrm{Cut} : M \times \mathcal{G} \rightarrow \mathcal{K}$ (\autoref{eq:cutLocusMap}), where $\mathcal{K}$ is the collection of compact sets in $M$. We topologize $\mathcal{K}$ with the Hausdorff metric. The main goal of this article is to prove the following.

\begin{theorem*}[\autoref{thm:cutLocusMapContinuous}]
    The map $\mathrm{Cut}$ is continuous.
\end{theorem*}

As a corollary to a lemma (\autoref{lemma:hausdorffConv2}) used in the proof of the above, we also show the continuous dependence of the cut time map on the $C^2$ perturbation of the metric (\autoref{cor:contCutTime}).
\medskip

\noindent \textbf{Conventions:} All manifolds are connected, and without boundary. Boldface letters, e.g., $\mathbf{u}, \mathbf{v},\mathbf{n},$ etc. will always denote tangent vectors, whereas letters in Fraktur font, e.g., $\mathfrak{g}, \mathfrak{h}$ etc. are reserved for Riemannian metrics. For any metric $\mathfrak{g}$, the induced distance, length, exponential map etc. are explicitly denoted with the subscript (or superscript) $\mathfrak{g}$. For a tangent vector $\mathbf{v}$, the unique geodesic with respect to $\mathfrak{g}$ with initial velocity $\mathbf{v}$ is denoted as $\gamma_{\mathbf{v}}^{\mathfrak{g}}(t)$. 

\medskip
\noindent \textbf{Organization of the paper:} In \autoref{sec:prelim}, we recall some basic definitions and results on cut locus, viscosity solutions of the eikonal equation, and convergence of metrics. Then, in \autoref{sec:continuity}, we prove the main theorem of this article.

\section{Preliminaries} \label{sec:prelim}
In this section we collect some preliminary results on cut locus in a Riemannian manifold from the viewpoint of viscosity solutions of the Eikonal equation. We also recall the notion of Hausdorff convergence of compact sets, and the $C^2$ convergence of Riemannian metrics, along with some immediate consequences.

\subsection{Cut Locus} We recall the definition of the cut locus of a point and more generally of a submanifold, while deferring to \cite{sakaiBook} for details.

Let us fix a complete Riemannian manifold $(M, \mathfrak{g})$. Given a $C^1$ curve $\gamma : [a, b] \rightarrow M$, the length is defined as 
\[L_{\mathfrak{g}}(\gamma) \coloneqq \int_a^b \left\lVert \dot \gamma \right\rVert _{\mathfrak{g}} dt.\]
Given $x, y \in M$, the induced distance is denoted as $d_{\mathfrak{g}}(x, y) = \inf L_{\mathfrak{g}}(\gamma)$, where the infimum is taken over the family of (piecewise) $C^1$ curves joining $x$ to $y$.

Given a vector $\mathbf{v} \in T_p M$, let us denote the unique $\mathfrak{g}$-geodesic with initial velocity $\mathbf{v}$ as $\gamma_{\mathbf{v}}^{\mathfrak{g}} : [0, \infty) \rightarrow M$. This lets us define the exponential map 
\begin{align*}
    \exp^{\mathfrak{g}} : T_p M &\rightarrow M \\
    \mathbf{v} &\mapsto \gamma^{\mathfrak{g}}_{\mathbf{v}}(1)
\end{align*}
By the Hopf-Rinow theorem, for any $q \in M$, we have a unit-speed $\mathfrak{g}$-geodesic, called a \emph{minimizer}, say, $\gamma : [0, L] \rightarrow M$ such that $\gamma(0) = p, \gamma(L) = q$ and $L = L_{\mathfrak{g}}(\gamma) = d_{\mathfrak{g}}(p, q)$. The \emph{cut locus} $\mathrm{Cut}(p, \mathfrak{g})$ consists of those points $q \in M$ beyond which a minimizer joining $p$ to $q$ fails to be distance minimizing from $p$.

Given a closed submanifold $N \subset M$, the normal bundle with respect to $\mathfrak{g}$ is denoted as 
\[TN^{\perp_{\mathfrak{g}}} \coloneqq \cup_{p \in N} \left\{ \mathbf{n} \in T_p M \;\middle|\; \mathfrak{g}(\mathbf{n}, \mathbf{v}) = 0, \; \text{for all } \mathbf{v} \in T_p N \right\}.\]
An \emph{$N$-geodesic} with respect to $\mathfrak{g}$ joining $N$ to $q$ is a $\mathfrak{g}$-geodesic joining a point in $N$ to $q$, with initial velocity in $TN^{\perp_{\mathfrak{g}}}$. A unit speed $N$-geodesic $\gamma : [0, L] \rightarrow M$ is said to be an \emph{$N$-minimizer} (with respect to $\mathfrak{g}$) joining $N$ to $q$ if $\gamma(0) \in N, \dot\gamma(0) \in TN^{\perp_{\mathfrak{g}}}$, and $\gamma(L) = q$, where $L = L_{\mathfrak{g}}(\gamma) = d_{\mathfrak{g}}(N, q)$. By the first variational principle, any unit-speed $\mathfrak{g}$-geodesic that minimizes the distance from $N$ to a point is an $N$-minimizer. Since $N$ is closed, the completeness of $\mathfrak{g}$ implies that given any $q \in M$ there exists an $N$-minimizer joining $N$ to $q$. We then define the cut locus $\mathrm{Cut}(N, \mathfrak{g})$ as the set of points $q \in M$, beyond which a minimizer joining $N$ to $q$ fails to be distance minimizing from $N$. The set of \emph{separating points} for $N$ is the set 
\[\mathrm{Sep}(N, \mathfrak{g}) \coloneqq \left\{ q \in M \;\middle|\; \text{there are at least two distance minimizing geodesics joining $N$ to $q$} \right\}.\]
It is well-known that $\mathrm{Sep}(N, \mathfrak{g}) \subset \mathrm{Cut}(N, \mathfrak{g})$, and moreover
\begin{equation}\label{eq:sepIsDenseInCut}
    \mathrm{Cut}(N, \mathfrak{g}) \coloneqq \overline{\mathrm{Sep}(N, \mathfrak{g})}.
\end{equation}
Given a unit vector $\mathbf{n}$ in the normal bundle $TN^{\perp_{\mathfrak{g}}}$ of $N$, the \emph{cut time} is defined as 
\begin{equation}\label{eq:cutTime}
    \rho_{\mathfrak{g}}(N, \mathbf{n}) \coloneqq \sup \left\{ t \;\middle|\;  d_{\mathfrak{g}}(N, \gamma_{\mathbf{n}}^{\mathfrak{g}}(t)) = t \right\}.
\end{equation}
As $\mathfrak{g}$ is complete, it follows that $\mathrm{Cut}(N, \mathfrak{g}) = \left\{ \gamma_{\mathbf{n}}^{\mathfrak{g}}(\rho_{\mathfrak{g}}(N, \mathbf{n})) \;\middle|\; \mathbf{n} \in TN^{\perp_{\mathfrak{g}}}, \; \left\lVert \mathbf{n} \right\rVert _{\mathfrak{g}} = 1\right\}$.

The \emph{injectivity radius} of $(M, \mathfrak{g})$ is defined as 
\begin{equation}\label{eq:injectivityRadius}
    \mathrm{Inj}(\mathfrak{g}) \coloneqq \inf_{p \in M} \mathrm{Inj}(p, \mathfrak{g}),
\end{equation}
where 
\[\mathrm{Inj}(p, \mathfrak{g}) \coloneqq \inf \left\{ \rho_{\mathfrak{g}}(p, \mathbf{v}) \;\middle|\; \mathbf{v} \in T_p M, \; \left\lVert \mathbf{v} \right\rVert _{\mathfrak{g}} = 1 \right\}.\]
If $M$ is compact, from the Whitehead convexity theorem \cite{whiteheadConvexity} it follows that $\mathrm{Inj}(\mathfrak{g}) > 0$. For any $0 < \delta < \mathrm{Inj}(\mathfrak{g})$, the set $S = S_{\mathfrak{g}}(p, \delta) \coloneqq \left\{ q \;\middle|\; d_{\mathfrak{g}}(p,q) = \delta \right\}$ is diffeomorphic to a sphere of dimension $\dim M - 1$. It is easy to see that
\begin{equation}\label{eq:cutLocusOfSphere}
    \mathrm{Cut}(S) = \left\{ p \right\} \cup \mathrm{Cut}(p, \mathfrak{g}).
\end{equation}
The \emph{diameter} of $(M, \mathfrak{g})$ is defined as the supremum
\begin{equation}\label{eq:diameter}
    \mathrm{Diam}(\mathfrak{g}) \coloneqq \sup_{p, q\in M} d_{\mathfrak{g}}(p, q),
\end{equation}
which is a maximum whenever $M$ is compact.

\subsection{Convergence of Compact Sets}
Let $(M, d)$ be a metric space. For arbitrary subsets $A, B \subset M$, the \emph{Hausdorff distance} between them is defined as 
\begin{equation}\label{eq:hausdorffDistance}
    d_H(A, B) \coloneqq \max\left\{ \sup_{a \in A} d(a, B), \; \sup_{b \in B} d(A, b) \right\},
\end{equation}
where $d(a, B) \coloneqq \inf_{b \in B} d(a, b)$, and $d(A, b) \coloneqq \inf_{a \in A}d(a,b)$. We shall need the following characterization for convergence with respect to the Hausdorff distance.

\begin{lemma}\label{lemma:hausdorffConvEquivalence}
    Suppose $M$ is compact, and $X_j, X \subset M$ are closed (hence compact) subsets. Then, $\lim_j d_H(X_j, X) = 0$ if and only if the following holds:
    \begin{enumerate}
        \item\label{lemma:hausdorffConvEquivalence:1} for any $x \in X$, there exists a sequence $x_j \in X_j$ with $\lim_j x_j = x$, and
        \item\label{lemma:hausdorffConvEquivalence:2} for any convergent sequence $x_{i_j} \in X_{i_j}$, we have $\lim_j x_{i_j} \in X$
    \end{enumerate}
\end{lemma}
\ifarxiv
\begin{proof}
    Suppose, $\lim_j d_H(X_j, X) = 0$. Let us prove (\hyperref[lemma:hausdorffConvEquivalence:1]{1}) first. Pick $x \in X$. We have, 
    \[0 \le d(X_j, x) \le \sup_{b \in X} d(X_j, b) \le d_H(X_j, X).\]
    Then, for each $n \ge 1$ we can get a strictly increasing sequence $N_n$ such that 
    \[d(X_j, x) < \frac{1}{n}, \quad \text{for all } j \ge N_n.\]
    Now, for each $N_n \le j < N_{n+1}$, we have $d(X_j, x) = \inf_{a \in X_j} d(a, x) < \frac{1}{n}$, and hence, we can pick $x_j \in X_j$ so that $d(x_j, x) < \frac{1}{n}$. For $1 \le j < N_1$, pick arbitrary $x_j \in X_j$. Then, we have $\lim_j x_j = x$, as required, showing (\hyperref[lemma:hausdorffConvEquivalence:1]{1}). Next, we prove (\hyperref[lemma:hausdorffConvEquivalence:2]{2}). Assume, $x_{i_j} \in X_{i_j}$ is a sequence converging to $y \coloneqq \lim_j x_{i_j}$. Now, for $\epsilon > 0$, there exists $N \ge 1$ such that 
    \[d(x_{i_j}, y) < \frac{\epsilon}{2}, \quad \text{and} \quad d_H(X_{i_j}, X) < \frac{\epsilon}{2}, \quad \text{for all } j \ge N.\]
    For $j \ge N$, we then have
    \[\sup_{a \in X_{i_j}} d(a, X) \le d_H(X_{i_j}, X) < \frac{\epsilon}{2},\]
    and therefore $d(x_{i_j}, X) < \frac{\epsilon}{2}$. Hence, there exists $y_j \in X$ such that $d(x_{i_j}, y_j) < \frac{\epsilon}{2}$. Thus, for $j \ge N$, we get 
    \[d(y_j, y) \le d(y_j, x_{i_j}) + d(x_{i_j}, y) < \frac{\epsilon}{2} + \frac{\epsilon}{2} = \epsilon.\]
    In other words, $y \in \bar{X} = X$, as $X$ is assumed to be closed. This proves (\hyperref[lemma:hausdorffConvEquivalence:2]{2})\\

    Let us now assume (\hyperref[lemma:hausdorffConvEquivalence:1]{1}) and (\hyperref[lemma:hausdorffConvEquivalence:2]{2}). Fix some $\epsilon > 0$. By compactness of $X$, we may assume that there are $y_1, \dots, y_K \in X$ such that 
    \[X = \bigcup_{i=1}^K B\left( y_i, \frac{\epsilon}{2} \right) = \bigcup_{i=1}^K \left\{ x \;\middle|\; d(y_i, x) < \frac{\epsilon}{2} \right\}.\]
    Then, from (\hyperref[lemma:hausdorffConvEquivalence:1]{1}), there exists $x^i_j \in X_j$ such that $\lim_j x^i_j = y_i$ for each $1 \le i \le K$. Consequently, we have $N \ge 1$ such that 
    \[d(X_j, y_i) < \frac{\epsilon}{2}, \quad \text{for all } j \ge N \; \text{and } 1 \le i \le K.\]
    Now, for any $y \in X$, we have some $1 \le i_0 \le K$ such that $d(y_{i_0}, y) < \frac{\epsilon}{2}$, and hence 
    \[d(X_j, y) \le d(X_j, y_{i_0}) + d(y_{i_0}, y) < \frac{\epsilon}{2} + \frac{\epsilon}{2} = \epsilon, \quad \text{for all } j \ge N.\]
    In other words, $\sup_{y \in X} d(X_j, y) \le \epsilon$ for $j \ge N$. As $\epsilon > 0$ is arbitrary, we get
    \[\lim_j \sup_{b \in X} d(X_j, b) = 0.\]
    Thus, we now need to show that $\lim_j \sup_{a \in X_j} d(a, X) = 0$. Let us assume the contrary. Then, for some $\epsilon > 0$, there exists an increasing sequence $i_j$ such that 
    \[\sup_{a \in X_{i_j}} d(a, X) > \epsilon.\]
    We can pick $x_{i_j} \in X_{i_j}$ such that $d(x_{i_j}, X) > \epsilon$. Now, since $M$ is compact, a subsequence of $x_{i_j}$, say, $x_{i_{j_k}}$ converges to some $y$. By (\hyperref[lemma:hausdorffConvEquivalence:2]{2}), we then have $y \in X$, which contradicts $d(x_{i_{j_k}}, y) \ge d(x_{i_{j_k}}, X) > \epsilon$. Hence, $\lim_j \sup_{a \in X_j} d(a, X) = 0$ as well. But then we have $\lim_j d_H(X_j, X) = 0$, concluding the proof.
\end{proof}
\fi

\subsection{Convergence of Riemannian Metrics} Let us briefly recall the Whitney $C^r$-topology on function spaces, we refer to \cite{guilleminGolubitsky, michorManfifoldsOfDiffMap} for details.

Given two topological spaces $X, Y$, on the collection of continuous maps $\mathcal{C}(X, Y)$ one can consider two natural topologies. The \emph{strong} topology has the basic open sets $\left\{ g \in \mathcal{C}(X, Y) \;\middle|\; g(X) \subset U \right\}$ for $U \subset X$ open. On the other hand, the \emph{weak} topology is the usual compact-open topology, where the basic open sets are of the form $\left\{ g \in \mathcal{C}(X, Y) \;\middle|\; g(K) \subset U \right\}$ for $K \subset X$ compact and $U \subset Y$ open. Note that the strong topology is not even Hausdorff, as it fails to separate any two surjective maps.

Now, suppose $(Y, d)$ is a metric space. Then, the weak topology on $\mathcal{C}(X, Y)$ is also known as the compact convergence topology. In particular, if $f_n \rightarrow f$ in the weak topology, then $f_n|_K \rightarrow f|_K$ uniformly for all compact set $K \subset X$. On the other hand, further assuming that $X$ is paracompact, a neighborhood basis of some $f \in \mathcal{C}(X, Y)$ in the strong topology is given as 
\[\left\{ g \in \mathcal{C}(X, Y) \;\middle|\; d(f(x), f(y)) < \epsilon(x), \; x \in X \right\},\]
where $\epsilon : X \rightarrow (0, \infty)$ is a continuous map. Consequently, $f_n \rightarrow f$ in the strong topology implies that there exists a compact set $K \subset X$, such that $f_n|_{X \setminus K} = f|_{X \setminus K}$ for $n$ large, and $f_n|_K \rightarrow f|_K$ uniformly.\medskip

Now, assume that $M, N$ are smooth manifolds. Then, the collection $J^r(M, N)$ of $r$-jets of maps $M \rightarrow N$ is itself a manifold. We have a natural injection 
\[j^r : C^\infty(M, N) \hookrightarrow \mathcal{C}\left( M, J^r(M, N) \right).\]
The induced topology on $C^\infty(M, N)$ given by the strong (resp. weak) topology on $\mathcal{C}\left( M, J^r(M, N) \right)$ is called the strong (resp. weak) Whitney $C^r$ topology. If $M$ is compact, the two topologies coincide.

Given a smooth vector bundle $E \rightarrow M$, one can construct the $r$-jet space of sections of $E$, and denote it as $E^{(r)} \subset J^r(M, E)$. Then, we have the injection $j^r : \Gamma(E) \hookrightarrow E^{(r)}$, which induces the strong (or weak) Whitney $C^r$ topology on the space $\Gamma E$ of smooth sections of $E$. Again, if $M$ is compact the two topologies are the same, and it is then a metric topology.\medskip

Let us now consider $E = T^*M \odot T^*M$, the symmetric tensor product of the cotangent bundles on a smooth manifold $M$ of dimension $n$. One can then identify the collection of smooth Riemannian metrics on $M$ as a subset $\mathcal{G} \subset \Gamma E$. Let us explain a $C^r$-neighborhood of a metric $\mathfrak{g} \in \mathcal{G}$ in the weak topology. Fix locally finite, closure compact, coordinate charts $\left\{ U_\alpha, \phi_\alpha \right\}$ covering $M$, and write $\mathfrak{g}$ as $\mathfrak{g}_{ij}^\alpha dx^i_\alpha \odot dx^j_\alpha$ on each $U_\alpha$. Set $\tilde{\mathfrak{g}}_{ij}^\alpha \coloneqq \mathfrak{g}_{ij}^\alpha \circ \phi_\alpha^{-1} : \phi_\alpha(U_\alpha) \subset \mathbb{R}^n \rightarrow \mathbb{R}$. Then, for a multi-index $\beta$ with order $\left\lvert \beta \right\rvert \le r$, we have the map $\partial_\beta \tilde{\mathfrak{g}}_{ij}^\alpha : \mathbb{R}^n \rightarrow \mathbb{R}$. Fix arbitrary positive numbers $\epsilon = \left\{ \epsilon_\alpha \right\}$. Then, another metric $\mathfrak{h} \in \mathcal{G}$ is $\epsilon$ close to $\mathfrak{g}$ in the $C^r$ topology if for each $\alpha$ and for each $i, j$ we have
    \[ \left\lvert \partial_\beta \tilde{\mathfrak{g}}_{ij}^\alpha(x) - \partial_\beta \tilde{\mathfrak{h}}_{ij}^\alpha(x)\right\rvert < \epsilon_\alpha, \quad x \in \phi_\alpha(U_\alpha), \; |\beta| \le r.\]
    In particular, $\mathfrak{g}_i \rightarrow \mathfrak{g}$ in the weak $C^r$ topology can be understood as the uniform $C^r$ convergence of the local coefficients (with some fixed coordinate system).\medskip

We observe that the metric convergence gives rise to the convergence in the norm and the length of a smooth curve.
\begin{lemma}\label{lemma:normConvergence}
    Let $\mathfrak{g}_i$ be a sequence of Riemannian metrics on $M$, converging in the weak $C^0$-topology to the metric $\mathfrak{g}$. Then, the induced norm $\left\lVert \cdot \right\rVert _{\mathfrak{g}_i}$ converges to $\left\lVert \cdot \right\rVert _{\mathfrak{g}}$ in the weak $C^0$-sense. Consequently, for any $C^1$ curve $\gamma : [0, 1] \rightarrow M$, the length $L_{\mathfrak{g}_i}(\gamma) \coloneqq \int_0^1 \left\lVert \dot \gamma \right\rVert _{\mathfrak{g}_i} dt$ converges to $L_{\mathfrak{g}}(\gamma) \coloneqq \int_0^1 \left\lVert \dot \gamma \right\rVert _{\mathfrak{g}} dt$.
\end{lemma}
\ifarxiv
\begin{proof}
    Let $U \subset M$ be a coordinate chart on $M$ with compact closure, denote the coordinate functions as $x^1,\dots ,x^n$. Fix some compact set $\mathcal{K} \subset TU$. For any $\mathbf{v} = \sum_i v^i \partial_{x^i} \in TU$, we have a continuous function $f(\mathbf{v}) = \sum_{i,j} |v_i v_j|$, which is bounded from above by some $C > 0$ on $\mathcal{K}$. Fix $\epsilon > 0$. Let us write $\mathfrak{g}_i \coloneqq \sum \mathfrak{g}_i^{jk}dx^j \odot dx^k, \mathfrak{g} \coloneqq \sum \mathfrak{g}^{jk}dx^j \odot dx^k$ on $U$. Since $\mathfrak{g}_i \rightarrow \mathfrak{g}$, we then have some $N \ge 1$ such that 
    \[\sum_{ij} \left\lvert \mathfrak{g}_i^{jk}(x) - \mathfrak{g}^{jk}(x) \right\rvert < \frac{\epsilon}{C}, \quad 1 \le j,k \le n, \; x \in U, \; i \ge N.\]
    Then, for any $\mathbf{v} \in \mathcal{K}$, we have 
    \[\left\lvert \mathfrak{g}_i(\mathbf{v}, \mathbf{v}) - \mathfrak{g}(\mathbf{v}, \mathbf{v}) \right\rvert \le \sum_{i,j} \left\lvert \mathfrak{g}_i^{jk} - \mathfrak{g}^{jk} \right\rvert |v_i v_j| \le C \sum_{i,j} \left\lvert \mathfrak{g}_i^{jk} - \mathfrak{g}^{jk} \right\rvert < \epsilon,\]
    for $i \ge N$. In other words, $\left\lVert \cdot \right\rVert^2_{\mathfrak{g}_i} \rightarrow \left\lVert \cdot \right\rVert^2_{\mathfrak{g}}$ uniformly on $\mathcal{K} \subset TU$. For an arbitrary compact set $\mathcal{K} \subset TM$, we can consider finitely many charts, covering the compact set $\pi(\mathcal{K}) \subset M$, where $\pi : TM \rightarrow M$ is the projection. Finally, as $x \mapsto \sqrt{x}$ is uniformly continuous on $[0, \infty)$, the proof follows immediately.
\end{proof}
\fi

\begin{remark}\label{rmk:sakaiC0Convergence}
    In \cite{ehrlichInjectivityCont,sakaiInjectivityCont}, the authors have used the notation $\mathfrak{g} \le \mathfrak{h}$ for two given metrics on $M$ to mean that $\mathfrak{g}(\mathbf{v}, \mathbf{v}) \le \mathfrak{h}(\mathbf{v}, \mathbf{v})$ for all $\mathbf{v} \in TM$. This leads to a weaker formulation of $C^0$-convergence when $M$ is compact. In particular, if $\mathfrak{g}_i \rightarrow \mathfrak{g}$ in the $C^0$-sense, then one can show that 
    \begin{equation}\label{eq:sakaiC0Conv}
        (1 - \epsilon_i) \mathfrak{g} \le \mathfrak{g}_i \le (1 + \epsilon_i) \mathfrak{g},
    \end{equation}
    for some real numbers $\epsilon_i \rightarrow 0$.
    \ifarxiv
    Note that if we prove the inequality for some $\mathbf{v} \ne 0$, we can deduce it for all $\lambda \mathbf{v}$ with $\lambda \ge 0$. Since the unit bundle $\left\{ \mathbf{v} \;\middle|\; \left\lVert \mathbf{v} \right\rVert _{\mathfrak{g}} = 1 \right\}$ is compact, using \autoref{lemma:normConvergence}, inductively we have $N_1 < N_2 < \dots $ such that 
    \[\left\lvert \mathfrak{g}_i(\mathbf{v}, \mathbf{v}) - \mathfrak{g}(\mathbf{v}, \mathbf{v}) \right\rvert < \frac{1}{k}, \quad \text{for all $\mathbf{v} \in TM$ with $\left\lVert \mathbf{v} \right\rVert _{\mathfrak{g}} = 1$, and for all $i \ge N_k$.}\]
    Set 
    \begin{equation*}
        \epsilon_i \coloneqq 
        \begin{cases}
            2 \max_{\left\lVert \mathbf{v} \right\rVert _{\mathfrak{g}} = 1} \left\lvert \mathfrak{g}_i (\mathbf{v},\mathbf{v}) - \mathfrak{g}(\mathbf{v}, \mathbf{v}) \right\rvert, \quad 1 \le i < N_1, \\
            \frac{1}{j}, \quad N_j \le i < N_{j + 1},
        \end{cases}
    \end{equation*}
    so that $\epsilon_i \rightarrow 0$. Then, for any $\mathbf{v}$, with $\mathfrak{g}(\mathbf{v}, \mathbf{v}) = 1$, we have
    \begin{align*}
        \left\lvert \mathfrak{g}_i(\mathbf{v}, \mathbf{v}) - \mathfrak{g}(\mathbf{v},\mathbf{v}) \right\rvert < \epsilon_i & \Rightarrow 1-\epsilon_i < \mathfrak{g}_i(\mathbf{v}, \mathbf{v}) < 1 + \epsilon_i \\
        & \Rightarrow (1 - \epsilon_i) \mathfrak{g}(\mathbf{v}, \mathbf{v}) < \mathfrak{g}_i(\mathbf{v}, \mathbf{v}) < (1 + \epsilon_i) \mathfrak{g}(\mathbf{v}, \mathbf{v}).
    \end{align*}
    Thus, \autoref{eq:sakaiC0Conv} follows.
    \else
    Indeed, this follows immediately from \autoref{lemma:normConvergence}, since the unit bundle $\left\{ \mathbf{v} \;\middle|\; \left\lVert \mathbf{v} \right\rVert _{\mathfrak{g}} = 1 \right\}$ is compact.
    \fi
\end{remark}

As an application, we get the convergence of the induced distance on a compact manifold.

\begin{lemma}\label{lemma:distanceConvergence}
    Let $\mathfrak{g}_i$ be sequence of Riemannian metrics on a compact manifold $M$, converging to the metric $\mathfrak{g}$ in the $C^0$-sense. Suppose, $p_i, q_i \in M$ are sequences of points, with $p = \lim p_i, q = \lim q_i$. Then, $\lim d_{\mathfrak{g}_i}(p_i,q_i) = d_{\mathfrak{g}}(p,q)$.
\end{lemma}
\ifarxiv
\begin{proof}
    Suppose $\gamma$ is a $\mathfrak{g}$-minimizer joining  $p$ to $q$ so that $d_{\mathfrak{g}}(p,q) = L_{\mathfrak{g}}(\gamma)$. Fix some $\epsilon > 0$. By \autoref{lemma:normConvergence}, we have $\lim L_{\mathfrak{g}_i}(\gamma) = L_{\mathfrak{g}}(\gamma)$. Thus, we have some $N^1_{\epsilon} \ge 1$ so that 
    \[L_{\mathfrak{g}_i}(\gamma) \le L_{\mathfrak{g}}(\gamma) + \epsilon = d_{\mathfrak{g}}(p,q) + \epsilon, \quad i \ge N_\epsilon^1.\]
    We also have some $N_\epsilon^2 \ge 1$ so that
    \[d_{\mathfrak{g}}(p_i, p) < \epsilon, \quad d_{\mathfrak{g}}(q_i, q) < \epsilon, \quad i \ge N_\epsilon^2.\]
    Lastly, by \autoref{lemma:normConvergence}, we have some $N_\epsilon^3$ so that 
    \[\left\lVert \mathbf{v} \right\rVert _{\mathfrak{g}_i} \le (1 + \epsilon) \left\lVert \mathbf{v} \right\rVert, \quad \mathbf{v} \in TM, \quad i \ge N_\epsilon^3.\]
    Set $N_\epsilon \coloneqq \max \left\{ N_\epsilon^1, N_\epsilon^2, N_\epsilon^3 \right\}$. Let $\eta_i : [0, d_{\mathfrak{g}}(p, p_i)] \rightarrow M$ be a $\mathfrak{g}$-minimizer joining $p_i$ to $p$, and $\zeta_i : [0, d_{\mathfrak{g}}(q, q_i)] \rightarrow M$ be a $\mathfrak{g}$-minimizer joining $q$ to $q_i$. Note that, for $i \ge N_\epsilon$ we have, 
    \[L_{\mathfrak{g}_i}(\eta_i) = \int_0^{d_{\mathfrak{g}}(p, p_i)} \left\lVert \dot \eta_i \right\rVert _{\mathfrak{g}_i} dt \le (1 + \epsilon) \int_0^{d_{\mathfrak{g}}(p, p_i)} \underbrace{\left\lVert \dot \eta_i \right\rVert _{\mathfrak{g}}}_1 dt = (1+\epsilon)d_{\mathfrak{g}}(p, p_i) < (1 + \epsilon)\epsilon,\]
    and similarly, $L_{\mathfrak{g}_i}(\zeta_i) < (1 + \epsilon)\epsilon$ as well. Then, for all $i \ge N_\epsilon$ we get
    \[d_{\mathfrak{g}_i}(p_i,q_i) \le L_{\mathfrak{g}_i}(\eta_i) + L_{\mathfrak{g}_i}(\gamma) + L_{\mathfrak{g}_i}(\zeta_i) < (1 +\epsilon)\epsilon + d_{\mathfrak{g}}(p,q) + \epsilon + (1 + \epsilon)\epsilon = d_{\mathfrak{g}}(p,q) + (3 + 2\epsilon)\epsilon\]
    which implies, $\lim\sup d_{\mathfrak{g}_i}(p_i,q_i) \le d_{\mathfrak{g}}(p,q) + (3 + 2 \epsilon)\epsilon$. As $\epsilon > 0$ is arbitrary, we get 
    \begin{equation}\label{eq:limsupDistance} \tag{$*_1$}
        \lim \sup d_{\mathfrak{g}_i}(p_i,q_i) \le d_{\mathfrak{g}}(p,q).
    \end{equation}

    Next, let us consider $\mathfrak{g}_i$-minimizers $\gamma_i : [0, L_i] \rightarrow M$ joining $p_i$ to $q_i$, where $L_i \coloneqq d_{\mathfrak{g}_i}(p_i,q_i)$. Since $\lim \sup L_i \le d_{\mathfrak{g}}(p, q) < \infty$, we have $L \coloneqq \sup L_i < \infty$. Extend $\gamma_i$ by constant to $[0, L]$, i.e., $\gamma_i(t) = q_i$ for $L_i \le t \le L$. Clearly $\left\{ \gamma_i \right\}$ is a uniformly bounded family, since $M$ is compact. As $\left\lVert \cdot \right\rVert _{\mathfrak{g}_i} \rightarrow \left\lVert \cdot \right\rVert _{\mathfrak{g}}$ uniformly, we have some $N_0 \ge 1$ so that
    \[\frac{1}{2}\left\lVert \mathbf{v} \right\rVert _{\mathfrak{g}} \le \left\lVert \mathbf{v} \right\rVert _{\mathfrak{g}_i} \le \frac{3}{2} \left\lVert \mathbf{v} \right\rVert _{\mathfrak{g}}, \quad \mathbf{v} \in TM,\]
    holds for $i \ge N_0$. Now, for $0 \le t_1 \le t_2 \le L_i$ we have, 
    \[d_{\mathfrak{g}}(\gamma_i(t_1), \gamma_i(t_2)) \le \int_{t_1}^{t_2} \left\lVert \dot \gamma(t) \right\rVert _{\mathfrak{g}} dt \le 2 \int_{t_1}^{t_2} \underbrace{\left\lVert \dot \gamma(t) \right\rVert _{\mathfrak{g}_i}}_1 dt = 2(t_2 - t_1).\]
    Since $\gamma_i$ is constant for $t \ge L_i$, we see that 
    \[d_{\mathfrak{g}}(\gamma_i(t_1))d_{\mathfrak{g}}(\gamma_i(t_1), \gamma_i(t_2)) \le 2\left\lvert t_2 - t_1 \right\rvert, \quad 0 \le t_1 \le t_2 \le L, \; \text{for all } i \ge N.\]
    Thus, the family $\left\{ \gamma_i \right\}_{i \ge N_0}$ is uniformly Lipschitz. By the Arzel\`{a}-Ascoli theorem, passing to a subsequence, we see that $\gamma_i \rightarrow \eta$ uniformly, for some curve $\eta : [0, L] \rightarrow M$. As the convergence is uniform, and the family $\left\{ \gamma_i \right\}_{i \ge N_0}$ is uniformly Lipschitz, it follows that $\eta$ is a Lipschitz (and hence, rectifiable) curve. In particular, the $\mathfrak{g}$-length of $\eta$ is defined as 
    \[L_{\mathfrak{g}}(\eta) \coloneqq \sup_{0 = t_0 < \dots < t_m = L} \sum_{i=1}^m d_{\mathfrak{g}}(\eta(t_{i-1}), \eta(t_i)).\] Clearly, $\eta(0) = \lim \gamma_i(0) = \lim p_i = p$ and $\eta(L) = \lim \gamma_i(L) = \lim q_i = q$, i.e., $\eta$ is a Lipschitz curve joining $p$ to $q$. We claim that $L_{\mathfrak{g}}(\eta) \le \lim \inf L_i = \lim \inf L_{\mathfrak{g}_i}(\gamma_i)$. Fix some $\epsilon > 0$. Then, we have some partition of $[0, L]$ so that $L_{\mathfrak{g}}(\eta) - \epsilon < \sum_{j=1}^m d_{\mathfrak{g}}(\eta(t_{j-1}), \eta(t_j))$. As $\gamma_i \rightarrow \eta$ pointwise, we have some $N_\epsilon \ge 1$ such that 
    \[d_{\mathfrak{g}}(\eta(t_j), \gamma_i(t_j)) < \frac{\epsilon}{2 m}, \quad 0 \le j \le m, \; i \ge N_\epsilon.\]
    We also assume, from \autoref{lemma:normConvergence} that
    \[\left\lVert \mathbf{v} \right\rVert _{\mathfrak{g}} \le (1 + \epsilon) \left\lVert \mathbf{v} \right\rVert _{\mathfrak{g}_i}, \quad \text{for all } \mathbf{v} \in TM, \; i \ge N_\epsilon.\]
    In particular, for $i \ge N_\epsilon$ we have
    \begin{align*}
        \sum_{j=1}^m d_{\mathfrak{g}}(\gamma_i(t_{j-1}), \gamma_i(t_j)) & \le \sum_{j=1}^m \int_{t_{j-1}}^{t_j} \left\lVert \dot\gamma_i(t) \right\rVert _{\mathfrak{g}}dt = \int_0^L \left\lVert \dot \gamma_i(t) \right\rVert _{\mathfrak{g}} dt \\
        &\le (1 + \epsilon) \int_0^L \left\lVert \dot \gamma_i(t) \right\rVert _{\mathfrak{g}_i} dt = (1 + \epsilon) d_{\mathfrak{g}_i}(p_i, q_i).
    \end{align*}
    Then, for all $i \ge N_\epsilon$,  by the triangle inequality of $d_{\mathfrak{g}}$ we get 
    \[L_{\mathfrak{g}}(\eta) - \epsilon < \sum_{j = 1}^m d_{\mathfrak{g}}\left( \gamma_i(t_{j-1}), \gamma_i(t_j) \right) + \epsilon \le (1 + \epsilon)d_{\mathfrak{g}_i}(p_i,q_i) + \epsilon,\]
    which implies $L_{\mathfrak{g}}(\eta) \le \lim \inf (1 + \epsilon) d_{\mathfrak{g}_i}(p_i,q_i) + 2\epsilon$. Since $\epsilon > 0$ is arbitrary, we have 
    \begin{equation}\label{eq:liminfDistance} \tag{$*_2$}
        d_{\mathfrak{g}}(p, q) \le L_{\mathfrak{g}}(\eta) \le \lim \inf d_{\mathfrak{g}_i}(p_i,q_i).
    \end{equation}
    But then from (\hyperref[eq:limsupDistance]{$*_1$}) and (\hyperref[eq:liminfDistance]{$*_2$}), we get $d_{\mathfrak{g}}(p,q) = \lim_i d_{\mathfrak{g}_i}(p_i,q_i)$, concluding the proof.
\end{proof}
\fi

We shall need the following.
\begin{theorem}\label{thm:sakaiInjectivityRadiusConv} \cite{ehrlichInjectivityCont,sakaiInjectivityCont}
    Suppose $\mathfrak{g}_i$ is a sequence of Riemannian metrics on a compact manifold, converging to the metric $\mathfrak{g}$ in the $C^0$-topology. Then, $\lim_i \mathrm{Diam}(\mathfrak{g}_i) = \mathrm{Diam}(\mathfrak{g})$. Furthermore, if $\mathfrak{g}_i \rightarrow \mathfrak{g}$ in $C^2$-topology, then $\lim_i \mathrm{Inj}(\mathfrak{g}_i) = \mathrm{Inj}(\mathfrak{g})$. 
\end{theorem}

We then have a stronger version of \autoref{lemma:distanceConvergence}.

\begin{lemma}\label{lemma:uniformConv}
    Suppose, $\mathfrak{g}_i$ is a sequence of Riemannian metrics on a compact manifold $M$, converging to the metric $\mathfrak{g}$ in $C^0$ topology, and $p_i \in M$ is a sequence of points converging to $p$. Set, $u_i \coloneqq d_{\mathfrak{g}_i}(p_i, \_)$ and $u \coloneqq d_{\mathfrak{g}}(p, \_)$. Then, $u_i \rightarrow u$ uniformly.
\end{lemma}
\begin{proof}
    It is clear from \autoref{lemma:distanceConvergence} that $u_i \rightarrow u$ pointwise. By \autoref{lemma:normConvergence}, we have some $N \ge 1$ such that 
    \[\left\lVert \mathbf{v} \right\rVert _{\mathfrak{g}_i} \le 2\left\lVert \mathbf{v} \right\rVert _{\mathfrak{g}}, \quad \mathbf{v} \in TM, \; i \ge N.\]
    Now, for $x, y \in M$, consider a $\mathfrak{g}$-minimizer $\gamma : [0, L] \rightarrow M$ joining $x$ to $y$, where $L = d_{\mathfrak{g}}(x, y)$. Then, for $i \ge N$ we have
    \[\left\lvert u_i(x) - u_i(y) \right\rvert \le d_{\mathfrak{g}_i}(x, y) \le L_{\mathfrak{g}_i}(\gamma) = \int_{0}^L \left\lVert \dot \gamma \right\rVert _{\mathfrak{g}_i} \le 2 \int_0^L \left\lVert \dot \gamma \right\rVert _{\mathfrak{g}} = 2L_{\mathfrak{g}}(\gamma) = 2 d_{\mathfrak{g}}(x, y).\]
    Thus, the family $\left\{ u_i \right\} _{i \ge N}$ is uniformly Lipschitz, with respect to the fixed metric $d_{\mathfrak{g}}$. Since $\lim \mathrm{Diam}(\mathfrak{g}_i) = \mathrm{Diam} (\mathfrak{g}) < \infty$ by \autoref{thm:sakaiInjectivityRadiusConv}, by taking $N$ larger, we may assume that $\mathrm{Diam}(\mathfrak{g}_i) < K \coloneqq \mathrm{Diam}(\mathfrak{g}) + 1$ for all $i \ge N$. Consequently, the family $\left\{ u_i \right\}_{i\ge N}$ is uniformly bounded as well, as $u_i(x) = d_{\mathfrak{g}_i}(p_i, x) \le \mathrm{Diam}(\mathfrak{g}_i) < K$.

    We now proceed by contradiction. Let us assume that $u_i \not\rightarrow u$ uniformly. Then, there exists an $\epsilon > 0$ and a subfamily, $\left\{ u_{i_j} \right\}$ of $\left\{ u_i \right\}_{i \ge N}$ such that $\sup_{x \in M}\left\lvert u_{i_j}(x) - u(x) \right\rvert \ge \epsilon$. But the family $\left\{ u_{i_j} \right\}$ is both uniformly Lipschitz and uniformly bounded. Hence, by the Arzel\`{a}-Ascoli theorem, we have a uniformly convergent subsequence $u_{i_{j_k}}$, which necessarily converges to $u$. This is a contradiction. Hence, $u_i \rightarrow u$ uniformly.
\end{proof}

\subsection{Viscosity Solution}
In this section, we recall the notion of viscosity solutions of the eikonal equation, we refer to \cite{lionsUserGuideViscosity} for details. Let $(M, \mathfrak{g})$ be a Riemannian manifold. Given an open set $\Omega \subset M$, consider the eikonal equation on $\Omega$
\begin{equation}\label{eq:eikonal}
    \left\lVert \nabla^{\mathfrak{g}} u \right\rVert _{\mathfrak{g}} = 1 \quad \left( \text{or equivalently,} \quad \left\lVert du \right\rVert _{\mathfrak{g}} = 1\right),
\end{equation}
where $\nabla^{\mathfrak{g}}$ is the gradient operator associated to $\mathfrak{g}$.

\begin{defn}\label{defn:viscositySoln}
    A continuous map $u : \Omega \rightarrow \mathbb{R}$ is said to be a \emph{viscosity} solution of \autoref{eq:eikonal} at $x_0 \in \Omega$ if it satisfies the following.
    \begin{itemize}
        \item $u$ is a viscosity \emph{sub-solution} at $x_0$ : for any $C^1$ function $\varphi$ near $x_0$, such that $u - \varphi$ attains a local maximum at $x_0$ (or equivalently, $u - \varphi \le 0$ and $u(x_0) = \varphi(x_0)$), we have $\left\lVert \nabla^{\mathfrak{g}} \varphi \right\rVert _{\mathfrak{g}} \le 1$ at $x_0$.
        \item $u$ is a viscosity \emph{super-solution} at $x_0$ : for any $C^1$ function $\psi$ near $x_0$, such that $u - \psi$ attains a local minimum at $x_0$ (or equivalently, $u - \psi \ge 0$ and $u(x_0) = \psi(x_0)$), we have $\left\lVert \nabla^{\mathfrak{g}} \psi \right\rVert _{\mathfrak{g}} \ge 1$ at $x_0$.
    \end{itemize}
\end{defn}

\begin{remark}\label{rmk:strictExtrema}
    Note that without loss of generality, we can ask for a \emph{strict} local maximum at $x_0$ in the definition above. Indeed, suppose $u - \phi$ attains a local maximum at $x_0$ for some $C^1$ map $\phi$ defined near $x_0$. Define $\varphi(x) \coloneqq \phi(x) + \left( d_{\mathfrak{g}}(x_0, x) \right)^2$. Clearly, $\varphi$ is $C^1$ in a small neighborhood of $U$, and furthermore, $d_{x_0} \varphi = d_{x_0} \phi$. Now, for $x \ne x_0$ we have, 
    \[(u-\varphi)(x) = (u - \phi)(x) - d_{\mathfrak{g}}(x_0, x)^2 \le u(x_0) - \underbrace{ d_{\mathfrak{g}}(x_0, x)^2 }_{ \gneq 0} < u(x_0),\]
    and $(u - \varphi)(x_0) = (u - \phi)(x_0)$. Thus, $u - \varphi$ attains a strict local maximum at $x_0$, while having $\left\lVert \nabla^{\mathfrak{g}}\phi \right\rVert = \left\lVert \nabla^{\mathfrak{g}} \varphi \right\rVert$ at $x_0$. Similarly, if $u - \psi$ attains a local minimum at $x_0$, it follows that $u - \varPsi$ attains a \emph{strict} local minimum at $x_0$, with $d_{x_0} \psi = d_{x_0} \varPsi$, where $\Psi(x) \coloneqq \psi(x) - \left( d_{\mathfrak{g}}(x_0, x) \right)^2$.
\end{remark}

Next, let us consider the Dirichlet problem 
\begin{equation}\label{eq:dirichletEikonal}
    \left\lVert d u \right\rVert _{\mathfrak{g}} = 1, \quad u|_{\partial \Omega} = 0.
\end{equation}
We have the following existence and uniqueness result for the viscosity solution of the above.

\begin{theorem}\cite[Theorem 3.1]{mantegazzaViscosityManifold} \label{thm:uniqueViscositySolution}
    Suppose $\Omega$ is bounded (which is the case since $M$ is compact). Then, the distance function $u \coloneqq d_{\mathfrak{g}}(\partial \Omega,\_ )$ on $\Omega$ is the unique viscosity solution (among the positive functions) for the Dirichlet problem \autoref{eq:dirichletEikonal}. The only other solution being $-u$.
\end{theorem}

Let us denote the sets
\begin{equation}\label{eq:insideCutSepLocus}
    \mathrm{Cut}^{\mathsf{in}}(\Omega, \mathfrak{g}) \coloneqq \mathrm{Cut}(\partial \Omega, \mathfrak{g}) \cap \Omega, \quad \mathrm{Sep}^{\mathsf{in}}(\Omega, \mathfrak{g}) \coloneqq \mathrm{Sep}(\partial \Omega, \mathfrak{g}) \cap \Omega.
\end{equation}
Then, the distance function $u \coloneqq d_{\mathfrak{g}}(\partial \Omega, \_)$ is not $C^1$ on points of $\mathrm{Sep}^{\mathsf{in}}(\Omega, \mathfrak{g})$, and it is smooth on $\Omega \setminus \mathrm{Cut}^{\mathsf{in}}(\Omega, \mathfrak{g})$. It follows from \autoref{eq:sepIsDenseInCut}, that the $C^1$ singular support of $u$ inside $\Omega$ is precisely the set $\mathrm{Cut}^{\mathsf{in}}(\Omega, \mathfrak{g})$. We shall need the following characterization. We provide a proof, though it may already be known to experts, while referring to \cite{mantegazzaViscosityManifold} for the terminology used within.
\begin{prop}\label{prop:characterizationOfSepPoint} 
    $x \in \mathrm{Sep}^{\mathsf{in}}(\Omega, \mathfrak{g})$ if and only if there exists $C^1$ function $\varphi$ near $x$ such that $u - \varphi$ attains a local maximum at $x$ and $\left\lVert d \varphi \right\rVert _{\mathfrak{g}} < 1$ at $x$ (or equivalently, $\left\lVert \nabla^{\mathfrak{g}}_x \varphi \right\rVert _{\mathfrak{g}} < 1$), where $u \coloneqq d_{\mathfrak{g}}\left( \partial \Omega, \_ \right)$.
\end{prop}
\begin{proof}
    Consider the Hamiltonian $H : TM \rightarrow \mathbb{R}$ given by $H(x, \mathbf{v}) = \left\lVert \mathbf{v} \right\rVert _{\mathfrak{g}}- 1 = \sqrt{\mathfrak{g}_x (\mathbf{v}, \mathbf{v})} - 1$. Observe that for any $\mathbf{u} \ne \mathbf{v} \in T_x M$ with $\left\lVert \mathbf{u} \right\rVert _{\mathfrak{g}} = 1 = \left\lVert \mathbf{v} \right\rVert _{\mathfrak{g}}$ we have the \emph{strict} inequality 
    \[H(x, \lambda \mathbf{u} + (1 - \lambda) \mathbf{v}) < \lambda H(x, \mathbf{u}) + (1 - \lambda) H(x, \mathbf{v}), \quad \lambda \in (0, 1).\]
    Indeed, it follows from the triangle inequality, which is strict since $\mathbf{u}$ and $\mathbf{v}$ being distinct unit vectors, cannot be collinear. Next, consider the set of \emph{supergradients}
    \[D^{+}u(x) \coloneqq \left\{ \nabla^{\mathfrak{g}}_x \varphi \;\middle|\; \text{$\varphi$ is $C^1$ near $x$, and $u - \varphi$ attains a local maximum at $x$.} \right\} \subset T_x M.\]
    If $u$ is differentiable at $x$, then $D^+ u(x) = \left\{ \nabla^{\mathfrak{g}}_x u \right\}$, whence $\left\lVert \nabla^{\mathfrak{g}}_x u \right\rVert = 1$. Moreover, as the distance function $u$ is \emph{locally semiconcave} in $\Omega$ \cite[Proposition 3.4]{mantegazzaViscosityManifold}, it follows that $D^+ u(x) \ne \emptyset$ for all $x \in \Omega$ \cite[Proposition 2.8]{mantegazzaViscosityManifold}. It is easy to see that $D^+ u(x)$ is convex: if for some $C^1$ functions $\varphi, \Psi$ near $x$, both $u - \varphi$ and $u - \varPsi$ attains a local maximum at $x$, then so does the function $\zeta_\lambda \coloneqq \lambda \varphi + (1 - \lambda) \varPsi$ for $\lambda \in [0, 1]$, and hence, $\nabla^{\mathfrak{g}}_x \zeta_\lambda = \lambda \nabla^{\mathfrak{g}}_x \varphi + (1 - \lambda) \nabla^{\mathfrak{g}}_x \varPsi \in D^{+} u(x)$. Thus, it follows from the strict convexity of $H$ that $D^+ u(x)$ contains at least two vectors if and only if there exists some $\varphi$ such that $u - \varphi$ attains a local maximum at $x$ and $\left\lVert \nabla^{\mathfrak{g}}_x \varphi \right\rVert _{\mathfrak{g}} < 1$.   Since $u$ is nondifferentiable precisely in $\mathrm{Sep}^{\mathsf{in}}(\Omega, \mathfrak{g})$, the claim follows immediately.
\end{proof}

We also have the following stability theorem.
\begin{theorem}\cite{lionsUserGuideViscosity} \label{thm:stabilityViscositySolution} 
    Let $\mathfrak{g}_i$ be a sequence of Riemannian metrics on a compact manifold $M$, converging to $\mathfrak{g}$ in the $C^0$ sense. Let $u_i$ be a viscosity solution defined on $\Omega \subset M$ to the eikonal equation $\left\lVert d f \right\rVert _{\mathfrak{g}_i} = 1$. Suppose $u_i \rightarrow u$ uniformly on compact sets. Then, $u$ is a viscosity solution on $\Omega$ to the eikonal equation $\left\lVert d f \right\rVert _{\mathfrak{g}} = 1$.
\end{theorem}
\begin{proof}
    Since $\mathfrak{g}_i \rightarrow \mathfrak{g}$ in the $C^0$ sense, we have $\left\lVert \cdot \right\rVert _{\mathfrak{g}_i} \rightarrow \left\lVert \cdot \right\rVert _{\mathfrak{g}}$ uniformly on $TM$ (\autoref{lemma:normConvergence}). Fix some $x_0 \in \Omega$. Suppose, $\varphi$ is a smooth function defined on a neighborhood $x_0 \in U \subset \Omega$ such that $u - \varphi$ attains a local maximum at $x_0$. As observed in \autoref{rmk:strictExtrema}, we assume that $u - \varphi$ attains a strict local maximum at $x_0$. Then, for $\epsilon > 0$ small, we have some $\delta_\epsilon > 0$ such that 
    \[(u - \varphi)(x) + \delta_\epsilon < (u - \varphi)(x_0), \quad \text{whenever } d_{\mathfrak{g}}(x_0, x) = \epsilon.\]
    Since $u_i \rightarrow u$ uniformly on compacts, we have some $N_\epsilon \ge 1$ such that 
    \[u(x) - \frac{\delta_\epsilon}{2} < u_i(x) < u(x) + \frac{\delta_\epsilon}{2}, \quad d_{\mathfrak{g}}(x_0, x) \le \epsilon, \; i \ge N_\epsilon.\]
    Hence, for any $i \ge N_\epsilon$ and for $d_{\mathfrak{g}}(x_0, x) = \epsilon$, we have
    \begin{align*}
        (u_i - \varphi)(x) = (u_i -u)(x) + (u - \varphi)(x) &< \frac{\delta_\epsilon}{2} + (u - \varphi)(x_0) - \delta_\epsilon \\
        &= (u - \varphi)(x_0) - \frac{\delta_\epsilon}{2} < (u_i - \varphi)(x_0).
    \end{align*}
    But then, $(u_i - \varphi)$ attains a local maximum at some $x_i$ with $d_{\mathfrak{g}}(x_0, x_i) < \epsilon$, whenever $i \ge N_\epsilon$. Taking $\epsilon \rightarrow 0$, we can then have a subsequence $x_i \in U$, with $x_i \rightarrow x_0$, so that $(u_i - \varphi)$ attains a local maximum at $x_i$. Now, since $u_i$ is a viscosity subsolution of $\left\lVert d f \right\rVert _{\mathfrak{g}_{i}} = 1$, we have 
    \[\left\lVert d_{x_{i}} \varphi \right\rVert _{\mathfrak{g}_{i}} \le 1.\]
    As $\left\lVert \cdot \right\rVert _{\mathfrak{g}_i} \rightarrow \left\lVert \cdot \right\rVert _{\mathfrak{g}}$ uniformly, taking $i \rightarrow \infty$, we have $\left\lVert d_{x_0}\varphi \right\rVert _{\mathfrak{g}} \le 1$. But then, $u$ is a subsolution. Arguing similarly, we can show that $u$ is a supersolution as well. Hence, $u$ is a viscosity solution to the eikonal equation $\left\lVert df \right\rVert _{\mathfrak{g}} = 1$.
\end{proof}

\section{Continuity of the Cut Locus map} \label{sec:continuity}
Let $M$ be a compact manifold, without boundary. Fix an auxiliary Riemannian metric $\mathfrak{h}$ on $M$. We consider the following collections:
\begin{itemize}
    \item $\mathcal{G} = \mathcal{G}(M) \coloneqq \left\{ \mathfrak{g} \in \Gamma T^*M \odot T^*M \;\middle|\; \text{$\mathfrak{g}$ is a $C^2$ Riemannian metric} \right\}$, and
    \item $\mathcal{K} = \mathcal{K}(M) \coloneqq \left\{ K \subset M \;\middle|\; \text{$K$ is compact} \right\}$.
\end{itemize}
Note that as $M$ is compact, every Riemannian metric $\mathfrak{g}$ on $M$ is complete. We consider the $C^2$ Whitney topology on $\mathcal{G}$. For $A, B \in \mathcal{K}$, define the \emph{Hausdorff metric} as 
\begin{equation}\label{eq:hausdorffMetric}
    d_H(A, B) \coloneqq \max \left\{ \sup_{a \in A} d_{\mathfrak{h}}(a, B), \; \sup_{b \in B} d_{\mathfrak{h}}(A, b) \right\} = \max \left\{ \sup_{a \in A} \inf_{b \in B} d_{\mathfrak{h}}(a, b), \; \sup_{b \in B} \inf_{a \in A} d_{\mathfrak{h}}(a, b) \right\}.
\end{equation}
The topology on $\mathcal{K}$ induced by $d_H$ is independent of the choice of the metric $\mathfrak{h}$. For $(p, \mathfrak{g}) \in M \times \mathcal{G}$, the cut locus $\mathrm{Cut}(p, \mathfrak{g})$ of $p$ with respect to the metric $\mathfrak{g}$ is a closed subset of $M$, and hence it is compact. In other words, we have a map
\begin{equation}\label{eq:cutLocusMap}
    \begin{aligned}
        \mathrm{Cut} : M \times \mathcal{G} &\rightarrow \mathcal{K} \\
        (p, \mathfrak{g}) &\mapsto \mathrm{Cut}(p, \mathfrak{g})
    \end{aligned}
\end{equation}
The main goal of this section is to prove the following.

\begin{theorem}\label{thm:cutLocusMapContinuous}
    The cut locus map $\mathrm{Cut} : M \times \mathcal{G} \rightarrow \mathcal{K}$ is continuous.
\end{theorem}

Since we are dealing with metric spaces, it is enough to prove the sequential continuity. Suppose $(p_i, \mathfrak{g}_i) \in M \times \mathcal{G}$ converges to $(p, \mathfrak{g})$. We shall require the following two lemmas.

\begin{lemma}\label{lemma:hausdorffConv1}
    Given any $x\in \mathrm{Cut}(p, \mathfrak{g})$, there exists a sequence $x_i \in \mathrm{Cut}(p_i, \mathfrak{g}_i)$ with $\lim_i x_i = x$.
\end{lemma}
\begin{proof}
    Since $\mathrm{Cut}(p, \mathfrak{g}) = \overline{\mathrm{Sep}(p, \mathfrak{g})}$, without loss of generality we may assume that $x \in \mathrm{Sep}(p, \mathfrak{g})$. As $\lim_i \mathrm{Inj}(\mathfrak{g}_i) = \mathrm{Inj}(\mathfrak{g})$ (\autoref{thm:sakaiInjectivityRadiusConv}), we can fix some $\delta > 0$ such that $\delta < \mathrm{Inj}(\mathfrak{g}_i)$ for all $i$ and $\delta < \mathrm{Inj}(\mathfrak{g})$. Without loss of generality, we assume $d_{\mathfrak{g}}(p, p_i) < \delta$ for $i \ge 1$ as well. Denote the sets
    \[\Omega_i \coloneqq \left\{ x \in M \;\middle|\; d_{\mathfrak{g}_i}(p_i, x) > \delta \right\}, \quad \Omega \coloneqq \left\{ x \in M \;\middle|\; d_{\mathfrak{g}}(p, x) > \delta \right\},\]
    and the distance functions 
    \[u_i \coloneqq d_{\mathfrak{g}_i}\left( \partial \Omega_i, \_ \right) = d_{\mathfrak{g}_i}(p_i, \_) - \delta, \quad u \coloneqq d_{\mathfrak{g}}\left( \partial \Omega, \_ \right) = d_{\mathfrak{g}}(p, \_) - \delta.\]
    It follows from \autoref{lemma:uniformConv} that $u_i \rightarrow u$ uniformly. Also, note that by \autoref{eq:cutLocusOfSphere} we have 
    \[\mathrm{Cut}^{\mathsf{in}}(\Omega_i, \mathfrak{g}_i) = \mathrm{Cut}(p_i, \mathfrak{g}_i), \quad \mathrm{Cut}^{\mathsf{in}}(\Omega, \mathfrak{g}) = \mathrm{Cut}(p, \mathfrak{g}),\]
    and similar relations hold for $\mathrm{Sep}(\cdot)$ as well. Then, by \autoref{prop:characterizationOfSepPoint}, we have a $C^1$ function $\varphi$ defined on a neighborhood $x \in U$ such that, $u - \varphi$ attains a local maximum at $x$ and 
    \[\left\lVert d_x \varphi \right\rVert _{\mathfrak{g}} < 1.\]
    By \autoref{rmk:strictExtrema}, we assume that $u- \varphi$ attains a \emph{strict} local maximum at $x_0$. Now, $u_i \rightarrow u$ uniformly. Hence, as in the proof of \autoref{thm:stabilityViscositySolution}, we have a sequence $x_i \in U$ such that $\lim x_i = x$ and $u_i - \varphi$ attains a local maximum at $x_i$. By \autoref{thm:uniqueViscositySolution}, $u_i$ is a viscosity sub-solution for the eikonal equation $\left\lVert df \right\rVert _{\mathfrak{g}_i} = 1$ in $\Omega_i$. Hence, we have $\left\lVert d \varphi \right\rVert _{\mathfrak{g}_i} \le 1$ at $x_i$. As $\varphi$ is $C^1$, from the convergence of the norm, we have
    \[\lim_i \left\lVert d_{x_i} \varphi \right\rVert _{\mathfrak{g}_i} = \left\lVert d_x \varphi \right\rVert _{\mathfrak{g}} < 1.\]
    Hence, for $i$ sufficiently large, we must have $\left\lVert d_{x_i} \varphi \right\rVert _{\mathfrak{g}_i} < 1$. But then again by \autoref{prop:characterizationOfSepPoint}, we have $x_i \in \mathrm{Sep}(p_i, \mathfrak{g}_i)$.  Hence, the claim follows.
\end{proof}

\begin{lemma}\label{lemma:hausdorffConv2}
    Suppose $x_{i_j} \in \mathrm{Cut}(p_{i_j}, \mathfrak{g}_{i_j})$ is a convergent sequence with $x \coloneqq \lim_j x_{i_j}$. Then, $x \in \mathrm{Cut}(p, \mathfrak{g})$.
\end{lemma}
\begin{proof}
    Firstly, to keep the notation light, we assume that $x_i \in \mathrm{Cut}(p_i, \mathfrak{g}_i)$. Next, in view of \autoref{eq:sepIsDenseInCut}, without loss of generality, we may assume that $x_i \in \mathrm{Sep}(p_i, \mathfrak{g}_i)$ so that $x \coloneqq \lim x_i$. We now proceed by contradiction and assume that $x \not \in \mathrm{Cut}(p, \mathfrak{g})$. 

    Set $L_i = d_{\mathfrak{g}_i}(p_i, x_i)$ and $L = d_{\mathfrak{g}}(p, x)$. From \autoref{lemma:distanceConvergence}, we have $L = \lim L_i$. Let $\gamma_{\mathbf{v}}^{\mathfrak{g}} : [0, L] \rightarrow M$ be the \emph{unique} $\mathfrak{g}$-minimizer joining $p$ to $x$, where $\mathbf{v} \in T_p M$ satisfies $\left\lVert \mathbf{v} \right\rVert _{\mathfrak{g}} = 1$. Since $x_i \in \mathrm{Sep}(p_i, \mathfrak{g}_i)$, we have $\mathbf{v}_i^1 \ne \mathbf{v}_i^2 \in T_{p_i} M$ with $\left\lVert \mathbf{v}_i^1 \right\rVert _{\mathfrak{g}_i} = 1 = \left\lVert \mathbf{v}_i^2 \right\rVert _{\mathfrak{g}_i}$, such that for $j=1,2$, the $\mathfrak{g}_i$-geodesics 
    \[\gamma_i^j \coloneqq \gamma_{\mathbf{v}_i^j}^{\mathfrak{g}_i}: [0, L_i] \rightarrow M\]
    are distinct $\mathfrak{g}_i$-minimizers joining $p_i$ to $x_i$. By a standard Arzel\`{a}-Ascoli type argument (see e.g. \cite[Theorem 5.16]{busemannBook}), passing to a subsequence, we have $\gamma_i^1 \rightarrow \eta$ uniformly, where $\eta : [0, L] \rightarrow M$ is a Lipschitz curve joining $p$ to $x$. But $L_{\mathfrak{g}}(\eta) = \lim L_{\mathfrak{g}_i}(\gamma_i^1) = \lim L_i = L$. Hence, $\eta=\gamma_{\mathbf{v}}^{\mathfrak{g}}$ and $\lim \mathbf{v}_i^1 = \mathbf{v}$. Note that since $x \not \in \mathrm{Cut}(p, \mathfrak{g})$, a sufficiently small extension of $\eta$ is still the unique minimizer joining $p$ to $x$.

    By \autoref{thm:sakaiInjectivityRadiusConv}, we have $\lim \mathrm{Inj}(\mathfrak{g}_i) = \mathrm{Inj}(\mathfrak{g}) > 0$. Hence, we may fix $\delta > 0$ small such that the following holds for $i$ sufficiently large:
    \[\delta < \frac{1}{2}\mathrm{Inj}(\mathfrak{g}_i), \quad \delta < \frac{1}{2}\mathrm{Inj}(\mathfrak{g}), \quad d_{\mathfrak{g}}(p, p_i) < \frac{\delta}{2}.\]
    In particular, the sets $S_i \coloneqq \left\{ z \in M \;\middle|\; d_{\mathfrak{g}_i}(p_i, z) = \delta \right\}$ and $S \coloneqq \left\{ z \in M \;\middle|\; d_{\mathfrak{g}}(p, z) = \delta \right\}$ are embedded submanifolds, diffeomorphic to spheres of dimension $\dim M - 1$. Set $q = \eta(\delta) = \gamma_{\mathbf{v}}^{\mathfrak{g}}(\delta) \in S$. We have neighborhoods $q \in U \subset M$ and $\delta\mathbf{v} \in \tilde{U} \subset T_p M$, such that 
    \[\exp_p^{\mathfrak{g}} : \tilde{U} \rightarrow U\]
    is a diffeomorphism. We assume that $\bar{U} \subset \left\{ z \in M \;\middle|\; d_{\mathfrak{g}}(q, z) < \frac{\delta}{2}\right\}$, and so, $\bar{U}$ is compact. Set $\tilde{V} \coloneqq \left\{ \frac{L}{\delta} \mathbf{u} \;\middle|\; \mathbf{u} \in \tilde{U }\right\}$, which is a neighborhood of $L\mathbf{v}$. Shrinking $\tilde{V}$ (and hence $\tilde{U}$ and $U$ as necessary), we have $\exp_p^{\mathfrak{g}}|_{\tilde{V}} : \tilde{V} \rightarrow V$ is a diffeomorphism, where $V$ is a neighborhood of $x$. This is possible because $x \not \in \mathrm{Cut}(p, \mathfrak{g})$. Thus, we have defined a diffeomorphism $\Phi : U \rightarrow V$ as the composition of three diffeomorphisms
    \begin{equation}\label{eq:definitionPhi}
        \begin{tikzcd}
            \Phi : U \arrow{rr}{\left( \exp_p^{\mathfrak{g}}\middle|_{\tilde{U}} \right)^{-1}} && \tilde{U} \arrow{rr}{\mathbf{u} \, \longmapsto \frac{L}{\delta}\mathbf{u}} && \tilde{V} \arrow{rr}{\left .\exp_p^{\mathfrak{g}}\right|_{\tilde{V}}} && V,
        \end{tikzcd}
    \end{equation}
    which maps a neighborhood of $q$ to a neighborhood of $x$. Next, set $q_i \coloneqq \gamma_i^1(\delta)$. Since $\lim q_i = \lim \gamma_i^1(\delta) = \eta(\delta) = q$ and $x = \lim x_i$, we have $x_i \in V$ and $q_i \in U$ for $i$ large. As $q_i \not \in \mathrm{Cut}(p_i, \mathfrak{g}_i)$, we can define $\Phi_i : U \rightarrow M$ as the compositions
    \begin{equation}\label{eq:definitionPhi_i}
        \begin{tikzcd}
            \Phi_i : U \arrow{rr}{\left( \exp_{p_i}^{\mathfrak{g}_i}\middle|_{\tilde{U}_i} \right)^{-1}} && \tilde{U}_i \arrow{rr}{\mathbf{u} \, \longmapsto \frac{L_i}{\delta}\mathbf{u}} && \tilde{V}_i \arrow{rr}{\left.\exp_{p_i}^{\mathfrak{g}_i}\right|_{\tilde{V}_i}} && M,
        \end{tikzcd}
    \end{equation}
    where the first two maps are diffeomorphisms onto the image. Clearly, $\mathbf{v}_i^1 \in \tilde{U}_i$ and $\Phi_i(q_i) = x_i$. Let us fix neighborhood $q \in W \subset U$ with $\bar{W} \subset U$, such that $q_i \in W$. Note that $\bar{W}$ is compact, as $\bar{U}$ is so. \medskip
    
    Firstly, we show that for $i$ sufficiently large, $\Phi_i$ maps $W$ in to $V$ (not-necessarily diffeomorphically). Suppose this is not the case. Then, we have a sequence, $x_{i_j} \in W, \mathbf{u}_{i_j} \in \tilde{U}_{i_j}$, with $\exp_{p_{i_j}}^{\mathfrak{g}_{i_j}}(\mathbf{u}_{i_j}) = x_{i_j}$, such that $y_{i_j} \coloneqq \Phi_{i_j}(x_{i_j}) \not \in V$. Passing to a further subsequence, $x_{i_j} \rightarrow x \in \bar{W} \subset U$. Let $\mathbf{u} \in \tilde{U}$ so that $\exp_p^{\mathfrak{g}}(\mathbf{u}) = x$. It follows that the minimizers $\gamma_{\mathbf{u}_{i_j}}^{\mathfrak{g}_{i_j}}$ joining $p_{i_j}$ to $x_{i_j}$ must converge, possibly passing to a subsequence, to the unique minimizer $\gamma_{\mathbf{u}}^{\mathfrak{g}}$ joining $p$ to $x$. Consequently, $\mathbf{u} = \lim \mathbf{u}_{i_j}$. Now, we have 
    \[y_{i_j} = \Phi_{i_j}(x_{i_j}) = \exp_{p_{i_j}}^{\mathfrak{g}_{i_j}}\left( \frac{L_i}{\delta} \mathbf{u}_{i_j}\right) = \gamma_{\mathbf{u}_{i_j}}^{\mathfrak{g}_{i_j}}\left( \frac{L_i}{\delta} \mathbf{u}_{i_j} \right) \rightarrow \gamma_{\mathbf{u}}^{\mathfrak{g}}\left( \frac{L}{\delta} \right) = \exp_p^{\mathfrak{g}}\left( \frac{L}{\delta} \mathbf{u} \right) = \Phi(x),\]
    which contradicts $y_{i_j} \not \in V$, as $\lim_j y_{i_j} = \Phi(x) \in V$.  Hence, for $i$ sufficiently large, we get the maps $\Phi_i : W \rightarrow V$.

    Next, we claim that for $i$ even larger, $\Phi_i : W \rightarrow V$ is an embedding. Consider the maps 
    \[\begin{aligned} \mathsf{h}_i : \tilde{U}_i &\rightarrow V\\
    \mathbf{u} &\mapsto \exp^{\mathfrak{g}_i}_{p_i}\left( \frac{L_i}{\delta} \mathbf{u} \right) \end{aligned}, \quad \begin{aligned}  \mathsf{h} : \tilde{U} &\rightarrow V \\
        \mathbf{u} &\mapsto \exp^{\mathfrak{g}}_{p}\left( \frac{L}{\delta} \mathbf{u} \right) \end{aligned},\]
    so that $\Phi_i = \mathsf{h}_i \circ \left( \exp_{p_i}^{\mathfrak{g}_i}|_{\tilde{U}_i} \right)^{-1}$ and $\Phi = \mathsf{h} \circ \left( \exp_p^{\mathfrak{g}}|_{\tilde{U}} \right)^{-1}$. Recall that $\mathsf{h}$ is a diffeomorphism. Assume, for the sake of contradiction, that for infinitely many $i$, we have $a_i^1, a_i^2 \in W$ such that $\Phi_i(a_i^1) = \Phi_i(a_i^2)$. We have $\mathbf{u}_i^1 \ne \mathbf{u}_i^2 \in \tilde{U}_i$ so that $\exp_{p_i}^{\mathfrak{g}_i}(\mathbf{u}_i^j) = a_i^j$ for $j=1,2$. Passing to a subsequence, $j = 1, 2$, we have $a_i^j \rightarrow a^j \in \bar{W} \subset U$, and hence, $\mathbf{u}^j \in \tilde{U}$ so that $\exp_p^{\mathfrak{g}}(\mathbf{u}^j) = a^j$. Arguing as in the previous paragraph, we see that, possibly for a subsequence, $\mathbf{u}_i^j \rightarrow \mathbf{u}^j$ for $j = 1,2$. Note that, since $\gamma_{\mathbf{u}_i^1}^{\mathfrak{g}_i}\left( \frac{L_{i}}{\delta} \right) = \Phi_i(a_i^1) = \Phi_i(a_i^2) = \gamma_{\mathbf{u}_i^2}^{\mathfrak{g}_i}\left( \frac{L_i}{\delta} \right)$ for all $i$, we have
    \[\mathsf{h}\left( \frac{L}{\delta} \mathbf{u}^1 \right) = \gamma_{\mathbf{u}^1}^{\mathfrak{g}}\left( \frac{L}{\delta} \right) = \lim_i \gamma_{\mathbf{u}_i^1}^{\mathfrak{g}_i}\left( \frac{L_i}{\delta} \right) = \lim_i \gamma_{\mathbf{u}_i^2}^{\mathfrak{g}_i}\left( \frac{L_i}{\delta} \right) = \gamma_{\mathbf{u}^2}^{\mathfrak{g}}\left( \frac{L}{\delta} \right) = \mathsf{h}\left( \frac{L}{\delta} \mathbf{u}^2 \right).\]
    As $\mathsf{h}$ is a diffeomorphism, we have $\mathbf{u}^1 = \mathbf{u}^2 = \mathbf{u} \in \tilde{U}$. Set $\mathbf{e}_i \coloneqq \frac{\mathbf{u}_i^2 - \mathbf{u}_i^1}{ \left\lVert \mathbf{u}_i^2 - \mathbf{u}_i^1 \right\rVert _{\mathfrak{g}_i} }$, so that $\left\lVert \mathbf{e}_i \right\rVert _{\mathfrak{g}_i} = 1$. Passing to a subsequence, we have $\mathbf{e}_i \rightarrow \mathbf{e} \in T_p M$. In particular, $\left\lVert \mathbf{e} \right\rVert _{\mathfrak{g}} = 1$, and so $\mathbf{e} \ne 0$. Without loss of generality, we can assume a priori that $V$ is a coordinate chart, with $\varphi : V \rightarrow \mathbb{R}^{n = \dim M}$ a coordinate map. Then, by the Taylor's theorem, we have 
    \[d(\varphi \circ \mathsf{h}_i)|_{\mathbf{u}_i^1}(\mathbf{u}_i^2 - \mathbf{u}_i^1) = - \int_0^1 (1-t) \left( d(\varphi \circ \mathsf{h}_i)|_{t \mathbf{u}_i^2 + (1-t)\mathbf{u}_i^1} - d(\varphi \circ \mathsf{h}_i)|_{\mathbf{u}_i^1} \right)(\mathbf{u}_i^2 - \mathbf{u}_i^1) dt.\]
    Dividing both sides by $\left\lVert \mathbf{u}_i^2 - \mathbf{u}_i^2 \right\rVert _{\mathfrak{g}_i}$ gives,
    \begin{equation}\label{eq:taylorIntegral}
        d(\varphi \circ \mathsf{h}_i)|_{\mathbf{u}_i^1}(\mathbf{e}_i) = - \int_0^1 (1-t) \left( d(\varphi \circ \mathsf{h}_i)|_{t \mathbf{u}_i^2 + (1-t)\mathbf{u}_i^1} - d(\varphi \circ \mathsf{h}_i)|_{\mathbf{u}_i^1} \right)(\mathbf{e}_i) dt.
    \end{equation}
    If we denote the geodesic flows associated to $\mathfrak{g}_i, \mathfrak{g}$ respectively as
    \[\mathfrak{G}_i^t : TM \rightarrow TM, \quad \mathfrak{G}^t :TM \rightarrow TM,\]
    it is easy to see that 
    \[\mathsf{h}_i = \pi \circ \mathfrak{G}_i^{\frac{L_i}{\delta}}|_{\tilde{U}_i}, \quad \mathsf{h} = \pi \circ \mathfrak{G}^{\frac{L}{\delta}}|_{\tilde{U}},\]
    where $\pi : TM \rightarrow M$ is the projection. Now, it follows from \cite[Lemma 1.6]{sakaiInjectivityCont} that $\mathfrak{G}_i^{\frac{L_i}{\delta}}\rightarrow \mathfrak{G}^{\frac{L}{\delta}}$ in the $C^1$ topology. This implies that the integrand in \autoref{eq:taylorIntegral} converges uniformly to the limit, as we can find a $C^1$ bound for $\pi$ and $\varphi$. Hence, taking $i \rightarrow \infty$ we get,
    \[d(\varphi \circ \mathsf{h})|_{\mathbf{u}}(\mathbf{e}) = - \int_0^1 (1-t)\left( d(\varphi\circ \mathsf{h})|_{\mathbf{u}} - d(\varphi \circ \mathsf{h})|_{\mathbf{u}} \right)(\mathbf{e}) dt = 0.\]
    But this contradicts the fact that $\mathsf{h}$ is a diffeomorphism. Hence, for $i$ large enough, we must have $\Phi_i : W \rightarrow V$ is injective. From \autoref{eq:definitionPhi_i}, it then follows that $\left. \exp^{\mathfrak{g}_i}_{p_i}\right|_{\tilde{W}_i}$ is injective, where $\tilde{W}_i$ is the image of $W$ under the first two diffeomorphisms. By \cite[Theorem 3.4]{warnerConjugateLocus}, we get $\left. \exp^{\mathfrak{g}_i}_{p_i}\right|_{\tilde{W}_i}$ has no critical point, whence it is a diffeomorphism. Thus,  $\Phi_i : W \rightarrow V$ is an embedding for $i$ large.\medskip

    Now, consider the points $q_i^\prime \coloneqq \gamma_i^2(\delta)$. Clearly $q_i^\prime \ne q_i$, as the geodesics $\gamma_i^1, \gamma_i^2$ intersects at time $L_i$ for the first time, and $L_i \ge \mathrm{Inj}(\mathfrak{g}_i) > \delta$. If $q_i^\prime \in W$, then $q_i^\prime \not \in \mathrm{Cut}(p_i, \mathfrak{g}_i)$ and clearly, $\left( \exp^{\mathfrak{g}_i}_{p_i}|_{\tilde{U}_i} \right)^{-1}(q_i^\prime) = \delta\mathbf{v}_i^2$. Consequently, $\Phi_i(q_i^\prime) = \exp_{p_i}^{\mathfrak{g}_i}(L_i \mathbf{v}_i^2) = \gamma_i^2(L_i) = x_i = \Phi_i(q_i)$, contradicting the injectivity of $\Phi_i$. Thus, $q_i^\prime \not \in W$ for $i$ large. On the other hand, by a Busemann type argument as above, we must have a subsequence of $\gamma_i^2$ converging to a curve joining $p$ to $x$, which is necessarily the unique minimizer $\eta = \gamma_{\mathbf{v}}^{\mathfrak{g}}$. This is a contradiction as $\gamma_i^2(\delta)$ cannot converge to $\eta(\delta)$. Hence, $x \in \mathrm{Cut}(p, \mathfrak{g})$, proving the claim.
\end{proof}

\begin{proof}[Proof of \autoref{thm:cutLocusMapContinuous}]
    For any sequence $(p_i, \mathfrak{g}_i) \in M \times \mathcal{G}$ converging to $(p, \mathfrak{g})$, let us denote $\mathcal{X}_i \coloneqq \mathrm{Cut}(p_i, \mathfrak{g}_i)$ and $\mathcal{X} \coloneqq \mathrm{Cut}(p, \mathfrak{g})$. We apply \autoref{lemma:hausdorffConvEquivalence} to show that $\lim d_H(\mathcal{X}_i, \mathcal{X}) = 0$. Clearly, \autoref{lemma:hausdorffConv1} justifies \autoref{lemma:hausdorffConvEquivalence} (\hyperref[lemma:hausdorffConvEquivalence:1]{1}), whereas \autoref{lemma:hausdorffConvEquivalence} (\hyperref[lemma:hausdorffConvEquivalence:2]{2}) follows from \autoref{lemma:hausdorffConv2}. This concludes the proof.
\end{proof}

We finish by observing an immediate corollary to \autoref{lemma:hausdorffConv2}. Recall that for a fixed point $p \in M$, the continuity of the cut time map $\mathbf{v} \mapsto \rho_{\mathfrak{g}}(p, \mathbf{v})$ is well known \cite[Proposition 4.1]{sakaiBook}. More generally, for a closed submanifold $N \subset M$, the cut time map $\mathbf{n} \mapsto \rho_{\mathfrak{g}}(N, \mathbf{n})$ is also continuous, where $\mathbf{n}$ varies in the normal bundle of $N$ \cite{basuPrasadCutLocus,bhowmickPrasadCutLocus}. We have the following, which improves up on the result from \cite{sakaiBook} for $M$ compact.

\begin{corollary}\label{cor:contCutTime}
    Suppose $\mathfrak{g}_i$ is a sequence of Riemannian metrics on a compact manifold $M$, converging in the $C^2$ norm to the metric $\mathfrak{g}$. Let $\mathbf{v}_i \in T_{p_i}M$ be vectors with $\left\lVert \mathbf{v}_i \right\rVert _{\mathfrak{g}_i} = 1$, converging to $\mathbf{v} \in T_p M$ (necessarily with $\left\lVert \mathbf{v} \right\rVert _{\mathfrak{g}} = 1$), where $p = \lim p_i$. Then, $\lim \rho_{\mathfrak{g}_i}(p_i, \mathbf{v}_i) = \rho_{\mathfrak{g}}(p, \mathbf{v})$.
\end{corollary}
\begin{proof}
    It follows from \autoref{thm:sakaiInjectivityRadiusConv} that $\lim \mathrm{Diam}(\mathfrak{g}_i) = \mathrm{Diam}(\mathfrak{g}) < \infty$. Hence, without loss of generality, we have $\mathrm{Diam}(\mathfrak{g}_i) < K \coloneqq \mathrm{Diam}(\mathfrak{g}) + 1$ for all $i$. But then $R_i \coloneqq \rho_{\mathfrak{g}_i}(p_i, \mathbf{v}_i) < K$ as well. Also, $R_i \ge \mathrm{Inj}(\mathfrak{g}_i) > 0$. In other words, $\left\{ R_i \right\} \subset [0, K]$ is a bounded sequence in $\mathbb{R}$. Suppose, $\left\{ R_{i_j} \right\}$ is a convergent subsequence, with $R \coloneqq \lim_j R_{i_j}$. Set $x_{i_j} \coloneqq \gamma^{\mathfrak{g}_{i_j}}_{\mathbf{v}_{i_j}}(R_{i_j})$ and $x \coloneqq \gamma^{\mathfrak{g}}_{\mathbf{v}}(R)$. Clearly, $x_{i_j} \in \mathrm{Cut}(p_{i_j}, \mathfrak{g}_{i_j})$ and also, $x = \lim x_{i_j}$. But then by \autoref{lemma:hausdorffConv2}, we have $x \in \mathrm{Cut}(p, \mathfrak{g})$. Consequently, $R = \rho_{\mathfrak{g}}(p, \mathbf{v})$. The claim then follows.
\end{proof}

\section*{Acknowledgments}
The first and third authors would like to express their gratitude to Prof. P. Albano for clarifying certain points from the article \cite{albanoCutLocusStability} through email correspondence. The first author was supported by the NBHM grant no. 0204/1(5)/2022/R\&D-II/5649. The second author was partially supported by Grant-in Aid for Scientific Research (C) (No. 21K03238), Japan Science for the Promotion of Science. The third author was supported by Jilin University.

\bibliographystyle{alphaurl}
\bibliography{ref}

\end{document}